\documentclass[11pt]{amsart}
\usepackage{amssymb, amscd}
\usepackage{latexsym,epsfig}
\usepackage[all]{xy}
\usepackage{pst-all}
\usepackage{pstcol}
\usepackage[normalem]{ulem}

\usepackage[
	hypertexnames=false,
	hyperindex,
	pagebackref,
	pdftex,
	breaklinks=true,
	bookmarks=false,
	colorlinks,
	linkcolor=blue,
	citecolor=red,
	urlcolor=red,
]{hyperref}

\usepackage[mathscr]{eucal}

\numberwithin{equation}{section}
\def\vandaag{\number\day\space\ifcase\month\or
 januari\or februari\or  maart\or  april\or mei\or juni\or  juli\or
 augustus\or  september\or  oktober\or november\or  december\or\fi,
\number\year}
\def\today{\ifcase\month\or
 Jan\or Febr\or  Mar\or  Apr\or May\or Jun\or  Jul\or
 Aug\or  Sep\or  Oct\or Nov\or  Dec\or\fi
 \space\number\day, \number\year}
 
\newtheorem{theorem}{Theorem}[section]
\newtheorem{lemma}[theorem]{Lemma}

\newtheorem{definition-lemma}[theorem]{Definition-Lemma}

\theoremstyle{definition}

\theoremstyle{remark}
\newtheorem{remark}[theorem]{\bf Remark}

\newcommand{\CC}{\mathbb C}

\newcommand{\EE}{\mathbb E}
\newcommand{\FF}{\mathbb F}

\newcommand{\QQ}{\mathbb Q}
\newcommand{\RR}{\mathbb R}

\newcommand{\ZZ}{\mathbb Z}
\newcommand{\A}{\mathcal A}
\newcommand{\R}{\mathcal R}

\newcommand{\cI}{\mathcal I}

\newcommand{\sF}{\mathscr {F}}
\newcommand{\cH}{\mathcal {H}}
\newcommand{\sH}{\mathscr {H}}

\newcommand{\sZ}{\mathscr {Z}}

\newcommand\Sp{\operatorname{Sp}}
\newcommand\ch{\operatorname{ch}}
\newcommand\Fr{\operatorname{Fr}}

\newcommand\an{\mathrm{an}}
\newcommand\gp{\mathrm{gp}}

\begin{document}
\title[]{Lifting Chern classes by means of Ekedahl-Oort strata}

\begin{abstract}
The moduli space $\A_g$ of principally polarized abelian varieties of genus $g$ 
is defined over $\ZZ$ and admits a minimal compactification $\A_g^*$, 
also defined over $\ZZ$. The Hodge bundle over $\A_g$ has its Chern classes in the Chow ring of $\A_g$ with  $\QQ$-coefficients.
We show that over  $\FF_p$, these  Chern classes naturally lift to  $\A_g^*$ 
and do so in the best possible way: despite the highly singular nature  
of $\A_g^*$ they are represented by algebraic cycles on $\A_g^*\otimes\FF_p$ 
which define elements in the  bivariant Chow ring. This is in contrast 
to the situation in the analytic topology, where these Chern classes have  canonical lifts to complex cohomology of the minimal compactification as Goresky-Pardon classes,  which are known to define nontrivial Tate extensions inside the mixed Hodge structure on this cohomology.
\end{abstract}

\author{Gerard van der Geer}
\address{Korteweg-de Vries Instituut, Universiteit van Amsterdam, Amsterdam, The Netherlands and Yau Mathematical Sciences Center, Tsinghua University Beijing, China.}
\email{g.b.m.vandergeer@uva.nl}

\author{Eduard Looijenga}\thanks{E.L. was supported by the Chinese National Science Foundation.}
\address{Yau Mathematical Sciences Center, Tsinghua University Beijing (China) and Mathematisch Instituut, Universiteit Utrecht (Nederland)}
\email{e.j.n.looijenga@uu.nl}


\maketitle

\medskip
\begin{section}{Introduction and statement of the main result}
Few objects in algebraic geometry have such a rich structure as the moduli space $\A_g$ of principally polarized
abelian varieties of dimension $g$. Its modular interpretation
makes it  a stack over ${\ZZ}$ and it comes as such with a rank $g$ vector bundle, the \emph{Hodge bundle} ${\EE^g}$
(which we may regard as the basic automorphic bundle over $\A_g$
in the sense that all other such over $\A_g$
are manufactured from it). Its determinant bundle $\det(\EE^g)$  is ample and when $g\ge 2$, the graded algebra of automorphic forms
$\oplus_{N=0}^\infty H^0(\A_g,(\det(\EE^g)^{\otimes N})$ is finitely generated so that its Proj  defines  a natural projective completion  $\A^*_g$ of  $\A_g$.
The complex-analytic space $\A_g^{\an}$
underlying $\mathcal{A}_g\otimes\CC$ has the familiar description as  the quotient of the Siegel upper half space $D_g$ of genus 
 $g$ by the integral symplectic group $\Sp_{2g}(\ZZ)$ and  $\mathcal{A}^*_g{}^{\an}$ is then the Satake-Baily-Borel compactification. 
Since $\A_g$ is a Deligne-Mumford stack, the (operational) Chow ring of  
$\A_g$ with $\QQ$-coefficients, $A_\QQ^\bullet(\A_g)$ is well-defined.
The Chern classes of $\EE^g$ generate a subalgebra $\R_g$ herein (we recall its presentation below). Since every automorphic bundle  over $\A_g$ is universally expressed via a Schur functor 
in terms of its Hodge bundle, $\R_g$ contains the Chern classes of all such bundles.  This is why we  refer to $\R_g$ as the \emph{tautological ring} of 
$\A_g$. 

The cycle map embeds this ring in 
$H^\bullet(\A_g^{\an}; \QQ)$ and its image can be characterized in several ways. One 
is that after tensoring with $\RR$, it is  the subalgebra representable by differential forms whose pull-backs to $D_g$
are $\Sp_{2g}(\RR)$-invariant (such forms are automatically closed).
 Charney and Lee \cite{charneylee:satake} had shown in 1983 that in the stable range (that is, for cohomological degree $<g$) these classes are liftable to $H^\bullet(\A_g^*{}^{\an}; \QQ)$, but Goresky and Pardon \cite{gp} proved in 2002
 that they admit in fact a \emph{natural} lift, provided  we use  the \emph{complex} cohomology of $\A^*_g{}^{\an}$. 
They raised the  question whether their lifts lie in the rational cohomology. The answer to that  question, given  by one of us \cite{L:gp}, 
is in general no. To see why, it is better to use the  Chern characters rather than the Chern classes, for  then the even indexed Chern characters are zero,  so that  the  issue regarding liftability only concerns the odd indexed ones. The answer is then that for $k=2r+1$ odd, the Goresky-Pardon lift $\ch_k^{\gp}(\EE^g)$ of $\ch_k(\EE^g)$ lies in the Hodge space  $F^kH^{2k}(\A^*_g{}^{\an})\subset H^{2k}(\A^*_g{}^{\an}; \CC)$. If we are also in the stable range  $0<k<g/2$, then, as we recall below,   it lies in fact in the complexification of a  mixed Tate substructure  of $H^{2k}(\A^*_g{}^{\an})$: an extension of $\QQ(-k)$ by $\QQ(0)$. This extension is nontrivial 
in the sense  that it is proportional to a standard nontrivial one whose invariant is given by $(2\pi\sqrt{-1})^{-k}\pi^{-k}\zeta (k)$. Since $k$ is odd, this implies that in this range,  
$\ch_k^{\gp}(\EE^g)$ will not even be a real  cohomology class.
 
We noted already back in 2015 that the situation is entirely different for $\A_g\otimes\FF_p$. For this, let us recall that Ekedahl and van der  Geer  \cite{E-vdG1} had proved  
that  $\R_g$ is then generated by the Ekedahl-Oort strata. Our observation at the time was that these strata intersect the boundary  of $\A_g^*\otimes\FF_p$ transversally with respect to its natural stratification (with ``minimal perversity''), which means that these classes  naturally lift  to $\ell$-adic cohomology classes  on $\A_g^*\otimes\FF_p$. 
We then realized that the notion of an $F$-zip, introduced by Moonen and Wedhorn in \cite{M-W} and the classifying space of such as introduced by 
Pink-Wedhorn-Ziegler \cite{P-W-Z1}
 make it fit into an even neater picture.
This classifying space of zips is an Artin stack, denoted  $[E_{\mathcal{Z}}\backslash \Sp_{2g}\otimes\FF_p]$
(we give more details below), which can be regarded as the characteristic  $p$ counterpart of the compact dual $\check{D}_g$ of the Siegel upper space $D_g$.
The Chow ring
$\check{\R}_g:=A_\QQ^\bullet([E_{\mathcal{Z}}\backslash \Sp_{2g}\otimes\FF_p])$ is isomorphic to the one of $\check{D}_g$. 
We have a natural morphism of Artin stacks $\A_g\to [E_{\mathcal{Z}}\backslash \Sp_{2g}\otimes\FF_p]$. It has  the property that it maps $\check{\R}_g$ onto $\R_g$.
Our main  observation now becomes:

\begin{theorem}\label{thm:main}
The morphism $\A_g\otimes\FF_p \to [E_{\mathcal{Z}}\backslash \Sp_{2g}\otimes\FF_p]$ naturally extends to the minimal compactification:
$\A_g^*\otimes\FF_p \to [E_{\mathcal{Z}}\backslash \Sp_{2g}\otimes\FF_p]$ and the induced ring homomorphism 
 $\check{\R}_g\to A^\bullet_\QQ( \A_g^*\otimes\FF_p)$ is an embedding.   
\end{theorem}

Here the ring $A^\bullet_\QQ( \A_g^*\otimes\FF_p)$ is Fulton's bivariant
Chow ring \cite{Fulton}. 
One may be tempted to call this image the tautological ring of $\A_g^*$, 
although (as was shown in  \cite{charneylee:satake}), the stable cohomology of the Baily-Borel compactification is larger than the
algebra generated by the $\lambda_i$ . 

\begin{remark}\label{rem:}
A natural analogue of this theorem can be stated for Shimura varieties $X$ of Hodge type, 
where the role of ${\R}_g$ is taken by the subalgebra $\R_G\subset A^\bullet_\QQ(X\otimes k)$, 
with $k$ a finite field, generated  by Chern classes of automorphic bundles. Here the subscript $G$ refers to the algebraic group that is part of the  data that give $X$ the structure of a 
Shimura variety. It is here implicit that we employ an integral model for $X$ which 
has good reduction over the prime with residue field $k$. Such models have been
constructed by Vasiu \cite{Vasiu} and Kisin \cite{Kisin}.  

The work of Pink, Wedhorn and Ziegler \cite{P-W-Z1,P-W-Z2} applies  to this setting:  we still have  
a  moduli stack of zips $[E_{\mathcal{Z}}\backslash G]$ and a classifying morphism 
$\zeta: X\to [E_{\mathcal{Z}}\backslash G]$, the fibres of which are the Ekedahl-Oort strata. 
The Chow $\QQ$-algebra of $[E_{\mathcal{Z}}\backslash G]$
(here denoted $\check{\R}_G$) is according to \cite[Thm.\ 2.4.4]{Brokemper} isomorphic to 
that of the compact dual  $\check{D}$.  If $\zeta$ is faithfully flat and surjective
and can be extended to a morphism $\tilde{\zeta}$ of a toroidal compactification of Faltings-Chai type, then
essentially the same proof shows that $\check{\R}_G$ embeds in the Chow algebra of the
toroidal compactification (see \cite{V-W,Zhang, Nie,K-M-S} for results in this direction).
The strata extend to the boundary and enjoy good intersection
properties with the boundary, see \cite[Thm.\ 6.1.6]{Boxer} and \cite{Lan-Stroh}.  
The morphism $\tilde{\zeta}$ factors through a morphism $\eta$ of the minimal
compactification to the stack $[E_{\mathcal{Z}}\backslash G]$ and we thus find in a way similar to
the case of ${\A}_g$ a copy of $\check{\R}_G$ in the
Chow algebra of the minimal  compactification $A^\bullet_\QQ(D_\Gamma^*)$.
We will confine ourselves however to the case $\A_g$.
\end{remark}

Let us note that Esnault and Harris \cite{Es_Ha2} recently proved a lifting property in the case of mixed characteristic, but on the level of $\ell$-adic cohomology. 
It would be interesting to see whether  their result can be lifted to  the  level of Chow algebras.

Recent work of Wedhorn-Ziegler \cite{We-Zi} and Goldring-Koskivirta \cite{G-K1}, \cite{G-K2} points towards a possible generalization to Shimura varieties of Hodge type.

\end{section}

\begin{section}{The Case $\mathcal{A}_g$}

\subsection{Review of the situation in characteristic zero.}
We let $\tilde{\A}_g$ be a toroidal compactification of $\A_g$ of
Faltings-Chai type 
and denote by $q: \tilde{\A}_g \to \A^{\ast}_g$ the natural projection.
The Hodge bundle ${\EE}^g$ on ${\A}_g$ extends to
$\tilde{\mathcal{A}}_g$ and this extension
is again denoted by ${\EE}^g$.

The analytic space of the complex fibre ${\A}^{\an}_g$ 
can be described in terms of
the Chevalley group $G=\Sp_{2g}$, 
the automorphism group of the standard
symplectic lattice ${\ZZ}^{2g}$ as
$G(\ZZ)\backslash D_g$  where
$D_g=G({\RR})/K$ is a bounded symmetric domain with $K$ a maximal 
compact subgroup.

Let us briefly review what is known about the Chow ring of the compact dual of $D_g$ in the  
more general case where $G$ is a reductive algebraic $\RR$-group whose symmetric space $D$
has the structure of a bounded symmetric domain. Then  the compact dual $\check{D}$ of $D$ 
is of the form $(G/P)(\CC)$ with $P$ a maximal parabolic subgroup of $G$.
We have a decomposition $G/P$ into Schubert cells: $G/P=\bigsqcup BwP/P$, where $w$ runs
over the elements  of the Weyl group $W$ of $G$, or rather (in order to keep the union disjoint), over a complete 
set $W^P$ of coset representatives for $W/W_P$, where  $W_P$ is the subgroup of $W$ associated to $P$.      
It is known that the Chow ring $A^{\bullet}(\check{D})$ has as an additive 
basis the classes
of the closures of Schubert cells (Schubert varieties) in $\check{D}$.
The ring structure on the Chow ring with $\QQ$-coefficients, $A^{\bullet}_{\QQ}(\check{D})$, is described by Borel (see \cite[p.\ 142, (28)]{Tamvakis}:
$$
A^{\bullet}_{\QQ}(\check{D})\cong  \mathcal{S}^{W_P}/\langle \mathcal{S}_{+}^W \rangle\,.
$$
Here $\mathcal{S}$ is the symmetric ${\QQ}$-algebra on the character group
of a Borel subgroup, $\mathcal{S}^{W_P}$ is the invariant part under
$W_P$ and $\langle \mathcal{S}_{+}^W \rangle$ is the ideal generated by 
$W$-invariant elements of positive
degree. 
In case the group is `special', e.g.\ for ${\rm GL}_n$ and $\Sp_{2n}$,
this isomorphism also holds for ${\ZZ}$-coefficients.  

In our case, where $G=\Sp_{2g}$, this graded $\QQ$-algebra  is isomorphic to 
$$
\check{\R}_g={\QQ}[u_1,\ldots,\ldots,u_g]/I,
$$
where $u_i$ has degree $i$ and $I$ is the ideal generated by the graded pieces of 
$$
(1+u_1+\cdots + u_g)(1-u_1+u_2- \cdots + (-1)^gu_g) -1. 
$$
So this gives a relation in every positive even degree $\le 2g$. Note that $\dim_\QQ\check{\R}_g=2^g$.

\smallskip
For a field $k$, the Chern classes $\lambda_i:=c_i({\EE}^g)$ in 
$A^i_{\QQ}(\tilde{\A}_g\otimes k)$ satisfy the same relation
as the $u_i$ in the Chow ring of $\tilde{\A}_g$ as the $u_i$:
$$
(1+\lambda_1+\cdots + \lambda_g)(1-\lambda_1+\cdots +(-1)^g\lambda_g)=1
$$
(see \cite{vdG, E-V}) and they generate a subring of the  Chow ring
$A_{\QQ}^{\bullet}(\tilde{\A}_g\otimes k)$ isomorphic to the rational Chow
ring of $\check{D}_g$. 
This extends the Hirzebruch-Mumford Proportionality
to the Chow rings.
This ring is called the tautological subring of 
$A^{\bullet}_{\QQ}(\tilde{\mathcal{A}}_g \otimes k))$ and denoted again by 
$\check{\R}_g$.
Its image in $A^{\bullet}_{\QQ}({\A}_g)$ under restriction via 
$j: {\A}_g\otimes k
 \hookrightarrow \tilde{\A}_g\otimes k$ is  ${\R}_g=\check{\R}_g/(\lambda_g)
\cong \check{\R}_{g-1}$.

\subsection{The Artin stack of zips.} 
We now restrict to characteristic $p$ and consider
${\mathcal A}_g\otimes {\FF}_p$ and 
$\tilde{{\mathcal A}}_g\otimes {\FF}_p$. 
The compact dual of Siegel space (or of any symmetric domain) has no 
obvious counterpart  in  positive characteristic.
But it turns out that there is a good substitute, \emph{viz.}\  the Artin stack of zips,  that  can take on that role
for our purposes. Its  origin  is the  
so-called Ekedahl-Oort stratification, introduced in \cite{Oort1}. 
As we will recall below, it has $2^g$ strata, and as was shown in \cite{vdG,E-vdG1}, each of these has the virtue that
the cycle class of its closure lies in the tautological subring.
For example,  we have the (closed) $p$-rank loci $V_f$ ($p$-rank $\leq f$
with $0 \leq f \leq g$)
with cycle classes $[V_f]=(p-1)(p^2-1)\cdots (p^{g-f}-1)\lambda_{g-f}$.
Thus the generators of $\check{\R}_g$ 
are represented by these effective cycles.

\subsubsection*{The basic definition}For a principally polarized abelian variety $X$ of dimension $g$ 
over a perfect field of characteristic $p>0$ the de Rham cohomology space $H^1_{\rm dR}(X)$
comes equipped with a non-degenerate alternating form.
The Frobenius operator  induces a $p$-linear endomorphism  of $H=H^1_{\rm dR}(X)$ whose kernel is
its Hodge subspace $H^0(X,\Omega_X^1)$. Both the kernel  and the 
image of this endomorphism are Lagrangian subspaces $F,F'$ of dimension $g$.
As we will see below, this structure (consisting of a symplectic vector space $H$ and a Frobenius-linear 
endomorphism  $\varphi$ of $H$
whose kernel and image is a Lagrangian subspace) has only 
finitely many isomorphism types. Such a  structure
is called a \emph{zip}  and was studied in \cite{M-W}. 
(To make the isomorphism type explicit one usually endows  
kernel and image with  filtrations by taking  preimages
and images of iterates of $\varphi$ and then  extends  these 
to self-dual filtrations on $H$ by adding their symplectic perps. 
This results in  a descending filtration 
(a refinement of the \emph{Hodge filtration}) $C^{\bullet}$,  
and an ascending filtration (a refinement of 
the \emph{conjugate filtration}) 
$D_{\bullet}$,  connected by the Cartier operator giving Frobenius-linear 
identifications $(C^{i}/C^{i+1})^{(p)} \cong D_i/D_{i-1}$. (The dimensions
of the intersections of these filtrations determine the isomorphism type. 
This will however not matter to us in what follows.) 

\subsubsection*{Moduli space and Schubert varieties}
In an evident manner we have defined a moduli space $\sZ(H)$ of all zip structures on $H$;
it is the moduli space of triples $(L_1,L_2,\varphi)$ with $L_1,L_2$ Lagrangian subspaces of $H$
and $\varphi: (H/L_1)^{(p)} {\buildrel \sim \over \to } L_2$ an isomorphism.
If $F(H)$ is the Grassmannian of Lagrangian subspaces of $H$ and $\sF_{F(H)}\to F(H)$ denotes its universal bundle, then
$\sZ(H)$ is an open subset in the  total space of the exterior tensor product bundle 
$$
\Fr_{p}^{*}(H\otimes {\mathcal O}_{F(H)}/\sF_{F(H)})\boxtimes \sF_{F(H)}
=\Fr_{p}^*(\sF_{F(H)})^\vee\boxtimes \sF_{F(H)}
$$ 
over $F(H)\times F(H)$, where $\Fr_p$ is the absolute Frobenius on $F(H)$. 
The group $G=\Sp (H)$ acts in an evident manner on $\sZ(H)$. We shall call the closure of a $G$-orbit in $\sZ(H)$ a \emph{Schubert variety}. 

There are  $2^g$ such Schubert varieties. This is based on the observation that  the relative position of a pair $(F,F')$ of Lagrangian 
subspaces (in other words, the $G$-orbit of such a pair) is given by a double coset of $G$: if $P$ (resp.\ $P'$) 
is the $G$-stabilizer of $F$ (resp.\ $F'$), then
the $g\in G$ for which $F=gF'$ make up the double coset $PgP'$, 
so that we get an element of $P\backslash G/P'$. We can identify this set of double cosets in terms of Weyl groups: if we choose
a Borel subgroup $B$ contained in $P$ with maximal torus $T$ and $N_G(T)$ (resp.\  $N_{P}(T) $) is the normalizer of $T$ in $G$ (resp.\ in $P$), then
$W=N_G(T)/T$ (resp.\  is $W_{P}:= N_{P}(T)/T$) 
is the Weyl group of the pair $(G,T)$ (resp.\  $(P,T)$) 
and it  is a standard fact of the theory of algebraic groups that the natural map
\[
W_{P}\backslash W\cong N_T(P)\backslash N_T(G) \to P\backslash G/P'
\] 
is a bijection. One finds that in our case $W_{P}\backslash W$ has  $2^g$ elements, and hence there are as many Schubert varieties.

\subsubsection*{The Artin stack of zip data} Let us make here the connection with the way this notion appears in the literature. The groups $P$  and $P'$ are maximal parabolic subgroups of $G$ 
whose Levi quotients $L_P$ resp.\ $L_{P'}$ can be identified 
with the general linear groups of $F$ (or of its dual $H/F$ for that matter) resp.\ $F'$. So an isomorphism $L_P\cong L'_{P'}$ can be understood as  giving an isomorphism $H/F\cong F'$ up to a scalar. Similarly, a Frobenius-linear map of $H/F$ onto $F'$ determines a  Frobenius  isogeny
$L_P\to L_{P'}$.  
We can formulate this in terms of $G$ only: in our setting a \emph{zip datum}  is given by a $4$-tuple $\mathcal{Z}=(G,P,P',\varphi)$, where 
$G={\rm Sp}_{2g}/{\FF}_p$, $P$ and $P'$ are maximal parabolic subgroups
of $G$ and $\varphi: L_P \rightarrow L_{P'}$ is an isogeny between their Levi quotients given by Frobenius. 
We form the  fibre product of $P$ and $P'$ over $L_{P'}$ (the former via the group homomorphism  $P\to L_P\xrightarrow{\varphi} L_{P'}$) in the category 
of algebraic groups: 
$$
E_{\mathcal{Z}}:=P\times_{L_{P'}} P' .
$$
This group acts on $G$ by $(p,q)\in E_{\mathcal{Z}}: g\mapsto pgq^{-1}$ and we can form the Artin stack $[E_{\mathcal{Z}}\backslash G]$.
Brokemper determined the Chow ring of 
 the stack $[E_{\mathcal{Z}}\backslash G]$ (which is essentially by definition the $G$-equivariant Chow ring of $\sZ(H)$).
He considers  in \cite{Brokemper} more generally the case of a 
connected group $G$ and an algebraic zip datum. 
Choose $g\in G$ such that $T':=gTg^{-1}\subset P'$.
If we identify $T$ resp.\ $T'$ with their images in $L_P$ 
resp.\ $L_{P'}$, then we can even 
arrange that $\varphi$ takes $T$ to $T'$, 
so that  we have defined an isogeny
$$
\widetilde{\varphi}: T \to T, \quad t\mapsto g^{-1}\varphi(t)g
$$
Then $\widetilde{\varphi}$ acts on $\mathcal{S}$, the symmetric algebra of the character group of $T$. The  Chow ring of the stack is (\cite[Thm.\ 2.4.4, page 27]{Brokemper})
$$
A^{\bullet}([E_{\mathcal{Z}}\backslash G]) = 
\mathcal{S}^{W_P}/\left( f-\widetilde{\varphi}(f): f \in \mathcal{S}_{+}^{W}
\right)
$$
In our case, this group is additively generated by the Schubert varieties as defined above. 

This Chow ring can be regarded as the ring of characteristic classes for symplectic vector bundles over $\FF_p$ endowed with a zip structure
for the following reason. If we have a symplectic vector bundle  $\sH$ over a base scheme $S$ (or stack, for that matter) over $\FF_p$ of  rank $2g$,  then the above construction yields the zip bundle $\sZ(\sH)$ over $S$, so that  to endow 
$\sH$ with a zip structure amounts to giving a section of $\sZ(\sH)/S$. This comes  with relative Schubert varieties and these define an 
embedding of $A^{\bullet}([E_{\mathcal{Z}}\backslash G])$ in Fulton's bivariant Chow ring $A^{\bullet} (\sZ(\sH))$ as a subalgebra, having these  relative Schubert varieties as additive generators. 
If a zip structure on $\sH$ has associated section $\sigma$, then  we may define its  ring of characteristic  classes as the image of  this subalgebra under  
$\sigma^*: A^{\bullet} (\sZ(\sH))\to A^\bullet(S)$. Note that when $\sigma$ has proper intersection with a given relative Schubert variety $Z$ in $\sZ(\sH)$, then the associated class $\sigma^*[Z]$ is represented by a \emph{specific algebraic cycle $\ge 0$} on $S$ defined over $\FF_p$; we shall refer to these as the \emph{Ekedahl-Oort cycles}. 

\subsection{Degenerations of zips}\label{subsection:degenerationofzips}
Let us for a moment return to our fixed symplectic vector space $H$ over $\FF_p$ and
suppose we are given an isotropic subspace $I\subset H$ over $\FF_p$. Then $H':=I^\perp/I$ is a symplectic  vector space over $\FF_p$ and 
we if assign to a Lagrangian subspace $F\subset H$  which  contains $I$ the subspace $F/I\subset H'$,  we get a bijection between
the  Lagrangian subspaces of $H$ containing $I$ and the Lagrangian subspaces of $H'$. 
Denote by $\sZ(H, I)\subset \sZ(H)$  the subscheme defined by the Frobenius-linear endomorphisms $\varphi$ of $H$
that  are zero on $I$,  preserve $I^\perp$, and  induce the Frobenius  on $H/I^\perp$.   
The kernel of $\varphi$ is sandwiched between $I$ and $I^\perp$ and the induced endomorphism $\varphi'$ of $H'$ defines an element of $\sZ(H')$, as
both its kernel  and  image are Lagrangian subspaces. The resulting  
morphism $\sZ(H, I)\to\sZ(H')$ is equivariant over the evident group homomorphism from the $\Sp(H)$-stabilizer of $I$ to
$\Sp (H')$ and this makes $\sZ(H, I)$ a torsor over a vector bundle on $\sZ(H')$.  
The preimage of a Schubert subvariety of $\sZ(H')$ is contained in a Schubert subvariety 
of $\sZ(H)$ of the same codimension. 
To be precise, every $\Sp(H)$-orbit in $\sZ(H)$ orbit meets $\sZ(H, I)$ transversally, 
and when this intersection is nonempty, then it is the preimage of a $\Sp (H')$-orbit in $\sZ(H')$.
Recall that the Schubert cells correspond bijectively to the elements of $W_P\backslash W$
with $P$ the stabilizer of a Lagrangian $F\subset H$ and similarly 
the Schubert cells of $\sZ(H')$ correspond
to  $W_{P'}\backslash W'$ with $P'$ the stabilizer of $F/I$ and $W'$ the Weyl group of ${\rm Sp}(H')$.
The map $\sZ(H, I)\to\sZ(H')$ is a stratified map corresponding to an 
embedding $\iota_I: W_{P'}\backslash W' \hookrightarrow W_P \backslash W$.

We use these observations to understand a class of degenerations of zips over a discrete valuation ring.
Let $R$ be a discrete valuation ring of finite type over $\FF_p$ with residue field $\kappa$ and field of fractions $K$.

Let $\cH$ denote a symplectic space of rank $2g$ over $R$
and $\cI \subset \cH$ an isotropic subspace over $R$ (so that 
$\cH':=\cI^\perp/\cI$ is a symplectic space over $R$).
If $H:= \kappa \otimes_R \cH$ with isotropic subspace $I$
(resp.\ $H':= \kappa\otimes_R \cH'$) 
denotes the fiber over the closed point, 
then  we have evident specialization maps $\sZ(\cH)\to \sZ(H)$
(resp.\   $\sZ(\cH')\to \sZ(H')$). 

Suppose given $\varphi \in \sZ(\cH,\cI)$ and assume that $\varphi_K$ belongs to
the Schubert cell with index $w$. We let $\varphi'$
be the image in $\sZ(\cH')$  with specialization $\varphi'_o \in \sZ(H')$.
The discussion above implies the following result.

\begin{lemma}\label{lemma:nicedegeneration}
If the element $\varphi_K$ belongs to the Schubert cell with index 
$w$ and $\varphi'_o$  to the Schubert cell $w'$,
then the specialization 
$\varphi_o$ belongs to the Schubert cell with index $\iota_I(w')=w$. 
\end{lemma} 

\subsection{Extension of the stratification across the Satake compactification}
By assigning to a  principally polarized abelian variety  
of dimension $g$ the isomorphism type of its zip
on its first de Rham cohomology space, we obtain a stratification
of the moduli space $\mathcal{A}_g \otimes {\FF}_p$, the \emph{Ekedahl-Oort stratification}. It is
is induced by a morphism of stacks
$$
\zeta: \mathcal{A}_g \to [E_{\mathcal{Z}}\backslash G]\, .
$$
This morphism is smooth (see \cite[Thm.\ 4.1.2]{Zhang}) 
and the fibres are the strata. 

This stratification can be extended to a 
toroidal compactification (of Chai-Faltings type)
$\tilde{{\mathcal A}}_g\otimes {\FF}_p$. 
The space $\tilde{\mathcal{A}}_g$ admits a stratification by torus rank:
if $q: \tilde{\mathcal{A}}_g \to \mathcal{A}^{*}_g$ is the canonical map
to the Baily-Borel compactification  and $\mathcal{A}^{*}_g=\sqcup_{i=0}^g \mathcal{A}_{g-i}$
is the standard decomposition, then the restriction of the Hodge bundle to
$\mathcal{A}_{g}^{\langle g-i \rangle}:=q^{-1}(\mathcal{A}_{g-i})$,
contains
a rank $g-i$ subbundle ${\EE}^{(g-i)}$ which is the pullback of the Hodge bundle
on $\mathcal{A}_{g-i}$.

The canonical extension of the de Rham complex is the logarithmic de Rham
complex where logarithmic singularities are allowed along the divisor added to
compactify the semi-abelian variety, 
cf.\ \cite[VI, Theorem 1.1, p.\ 195]{F-C}.
The logarithmic de Rham sheaf
$$
{\mathcal H}^1:=R^1\pi_{*}(\Omega^{\bullet}_{\tilde{\mathcal{X}}_g/\tilde{\mathcal{A}}_g}(\log))
$$
extends the de Rham sheaf 
$\mathcal{H}_{\rm dR}^1(\mathcal{X}/\mathcal{A}_g)$.
On $\tilde{\mathcal{A}}_g\otimes {\FF}_p$ it comes again with
two filtrations forming a zip.
In fact, the morphism $\zeta$ can be 
extended to a morphism 
$\tilde{\zeta}: \tilde{\mathcal{A}}_g \otimes {\FF}_p 
\to [E_{\mathcal{Z}}\backslash G]$
which is again smooth as can be seen by using \cite[Lemma 5.1]{E-vdG1} or \cite{Boxer}, see
also below. 
The closed strata on $\tilde{A}_g \otimes {\FF}_p$ 
are the closures of
the strata on $\mathcal{A}_g \otimes {\FF}_p$.

\bigskip
The Ekedahl-Oort stratification on $\tilde{\mathcal{A}}_g \otimes {\FF}_p$ 
intersects the boundary strata transversally
as we will now explain. The reason is that the Ekedahl-Oort stratification
is defined by the action of Frobenius and Verschiebung acting on the 
logarithmic de Rham cohomology $H^1_{\rm dR}$ of a semi-abelian variety and on the toric
part this action is essentially trivial.

We consider a semi-abelian variety ${\mathcal G}^0$ over $S={\rm Spec}(R)$
with $R$ a discrete valuation ring of finite type over ${\FF}_p$.
We assume that the generic fibre is abelian and the special fibre is the
N\'eron model of a semi-abelian variety of torus rank $r$.
We let ${\mathcal G}/S$ be a toroidal compactification of ${\mathcal G}^0$ 
of Faltings-Chai type.
It can be obtained via the action on a semi-abelian variety
$\tilde{\mathcal G}$ over $S$ by a group
of periods $\iota: Y \to \tilde{\mathcal G}(S)$ with $Y$
free abelian of rank $r$.
Here the semi-abelian variety $\tilde{\mathcal G}$ is an extension 
$0 \to T \to \tilde{\mathcal G} \to A \to 0$
of an abelian scheme $A/S$ by a split torus $T/S$ of rank $r$.
In this case the logarithmic de Rham cohomology can be described
with the help of universal vector extensions, that is, extensions of
group schemes by vector group schemes. We refer to \cite{F-C} pages 81--86
for a description. The universal vector extension $E_{\tilde{\mathcal G}}$
of $\tilde{\mathcal G}$ is a vector group extension 
$$
0 \to L_{\tilde{\mathcal G}} \to E_{\tilde{\mathcal G}} \to \tilde{\mathcal G} \to 0
$$
that is canonically isomorphic to the pullback under $\tilde{\mathcal G} \to A$
of the universal vector extension $0 \to L_A \to E_A \to A\to 0$ of $A$,
where $L_A={\rm Lie}(A^{\vee}/S)^{\vee}$ is the sheaf of invariant relative
$1$-forms on the dual abelian variety $A^{\vee}$ of $A$.
For the quotient construction we need an equivariant form of this,
that is, we need in addition a lifting of the homomorphism $Y \to \tilde{\mathcal G}(S)$
to $Y \to E_{\tilde{\mathcal G}}(S)$. Then $Y$ acts via translation.

The dual of the logarithmic de Rham cohomology ${\mathcal H}^1$ is the $Y$-equivariant
Lie-algebra of the universal vector extension of $\tilde{\mathcal G}$.
By the toroidal construction as in \cite[Ch.\ VI]{F-C} this Lie-algebra has a weight filtration with
subquotients the Lie algebra $L_T$ of $T$, the homology of the abelian variety $A$
and $Y$. The ranks are $r$, $2g-2r$ and $r$. The subspace of rank $2g-r$ of  ${\mathcal H}^1$
will be denoted by $I^{\bot}$ and its orthogonal complement by $I$. 
We can identify $I$ with the invariant differentials
of the torus $T$. Then $I$ is isotropic and contained in the kernel of 
Frobenius. We are thus in the situation described above in subsection 
\ref{subsection:degenerationofzips}. 
Since $I$ is contained in the kernel of Frobenius
the isomorphism type of the zip on the special fibre of
${\mathcal H}^1$ depends only on the zip of the de Rham cohomology
of the abelian part. 
We can apply Lemma \ref{lemma:nicedegeneration}
to conclude that the closures of the strata on ${\A}_g \otimes {\FF}_p$ are the strata
on $\tilde{\A}_g \otimes {\FF}_p$ and by induction that the intersection
with the boundary strata is proper. Indeed, with the notation used there,
if $\varphi_K \in \sZ(\cH_K,\cI_K)$  and $\varphi'_K$ belongs to the Schubert
cycle with index $w' \in W_{P'}\backslash W'$ then $\varphi_K$ extends uniquely
to $\varphi$ with Schubert index $\iota_I(w')$.

\begin{remark}
The valuation of the torus part of the periods defines a ${\ZZ}$-valued bilinear 
form on $Y$ which we can see as the analogue of the monodromy operator of Hodge theory. 
Its invariant part defines a subspace $I^{\bot}$ of dimension $2g-r$  in the
special fibre of the logarithmic de Rham cohomology over $S$.
(One might view it as associated to the Dieudonn\'e module of the kernel of multiplication
by $p$ on the semi-abelian special fibre of ${\mathcal G}$.)
\end{remark}

\bigskip

We thus see that the map 
$\tilde{\zeta}: \tilde{\mathcal{A}}_g \otimes {\FF}_p\to 
[E_{\mathcal{Z}}\backslash G]$ factors through a map 
$$
\begin{xy}
\xymatrix{
\tilde{\mathcal{A}}_g \otimes {\FF}_p \ar[rr]^{\tilde{\zeta}}
\ar[rd]^{q}  && [E_{\mathcal{Z}}\backslash G] \\
& \mathcal{A}_g^{\ast}\otimes {\FF}_p \ar[ru]^{\eta} & \\
}
\end{xy}
$$

The morphism $\tilde{\zeta}: \tilde{\mathcal{A}}_g \otimes {\FF}_p
\to [E_{\mathcal{Z}}\backslash G]$ induces a homomorphism of Chow rings
$$
A^{\bullet}_{\QQ}([E_{\mathcal{Z}}\backslash G]) 
\to A^{\bullet}_{\QQ}(\tilde{\mathcal{A}}_g \otimes {\FF}_p)
$$
and it induces an isomorphism
$A^{\bullet}_{\QQ}([E_{\mathcal{Z}}\backslash G]) \cong \check{\R}_g$.
Indeed, the closed Ekedahl-Oort strata on
$\tilde{A}_g \otimes {\FF}_p$
are effective cycles with non-zero classes.

\begin{proof}[Proof of Theorem \ref{thm:main}]
The image under push forward via 
$q: \tilde{\A}_g\to {\A}_g^{\ast}$ 
of 
$\lambda_i\in  A^i_{\QQ}(\tilde{{\A}}_g)$ 
in the Chow cohomology group $A^{i}_{\QQ}({\A}_g^{\ast})$ 
is independent of the chosen toroidal compactification, see \cite[Def-Prop.\ 3.1]{E-vdG2}.
Thus these define classes  $\lambda_i^{\prime}$ in $A^i_{\QQ}({\A}_g^*)$.
On the other hand we have the generators $\underline{\lambda}_i$ of the
Chow ring of the stack 
$[E_{\mathcal{Z}}\backslash G]$ and via the map 
$\eta: {\A}_g^{\ast} \otimes {\FF}_p
\to [E_{\mathcal{Z}}\backslash G]$ these act
as bivariant classes  by cap product 
$\cap \underline{\lambda}_i: 
A_{k}({\A}_g^{\ast}\otimes {\FF}_p) \to A_{k-i}({\A}_g^{\ast}\otimes {\FF}_p)$
on the Chow groups of ${\A}^{\ast}_g \otimes {\FF}_p$. These satisfy 
$\tilde{\zeta}^*(\underline{\lambda}_i)=\lambda_i$.
By \cite[17.1]{Fulton} and the projection formula
(\cite[p.\ 323]{Fulton}) we have
$$
\cap \underline{\lambda}_i \, (q_*(c))= q_*(\lambda_i \cdot c)
=\lambda_i^{\prime}\, q_*(c)
$$
for all $c \in A_k(\tilde{\A}_g)$. 
This enables us identify the 
bivariant classes $\underline{\lambda}_i$ with the $\lambda_i^{\prime}$. 
It thus gives rise to a diagram
$$
\begin{xy}
\xymatrix{
A^{\bullet}_{\QQ} ([E_{\mathcal{Z}}\backslash G]) \ar@{-->}[rrd] \ar[rr]^{\tilde{\zeta}^*} &&
\check{\R}_g \ar@{^{(}->}[r] \ar[d]_{\cong} & A^{\bullet}_{\QQ}(\tilde{\mathcal{A}}_g\otimes {\FF}_p)\ar[d]^{q_*} \\
&& \check{\R}_g' \ar@{^{(}->}[r] & A^{\bullet}_{\QQ}(\mathcal{A}_g^{\ast} \otimes {\FF}_p)
 \\}
\end{xy}
$$ 
\end{proof}

\begin{remark}\label{rem:}
In the end the argument is based on the observation that {\sl all the tautological classes
$\lambda_i$ have an effective representative on $\tilde{\mathcal{A}}_g
\otimes {\FF}_p$ that intersects  the boundary properly.}
This fails to be so in characteristic zero, although  it is then true for the ample $\lambda_1$,
and hence for any power on $\lambda_1$, like $\lambda_2=(1/2)\lambda_1^2$. But this is not
so for $\lambda_3$. This seems related to the question of whether for a given field $k$ 
the space $\mathcal{A}_g\otimes k$ contains complete
subvarieties of codimension $g$. 
For $g=3$ every complete subvariety of $\mathcal{A}_3\otimes k$ has as class a
multiple of $\lambda_3$.
Conversely, an effective representative for $\lambda_g$ transversal to the
boundary of $\tilde{\A}_g\otimes k$ does not intersect the boundary because $\lambda_g^2=0$,
hence yields a complete subvariety of codimension $g$ .
\end{remark}
\section*{Acknowledgements}
We thank Luc Illusie and Gerd Faltings for correspondence and the referee for useful suggestions.

\end{section}

\end{document}